\documentclass[12pt]{amsart}

\usepackage{amsmath}
\usepackage{amsfonts}
\usepackage{amsthm}

\usepackage[cmtip,arrow]{xy}
\usepackage{pb-diagram, pb-xy}

\newtheorem{theorem}{Theorem}

\newtheorem{lemma}{Lemma}
\newtheorem{proposition}{Proposition}
\newtheorem{definition}{Definition}

\newtheorem{remark}{Remark}

\newcommand{\G}{\mathcal{G}}

\renewcommand{\L}{\mathcal{L}}

\begin{document}
\title{Gorenstein flat dimension of complexes.}
\author{Alina Iacob}

\address{Department of Mathematical Sciences\\
 Georgia Southern University\\
     Statesboro, GA 30458 USA\\
     Email: aiacob@georgiasouthern.edu}



\maketitle

\begin{abstract}
We define a notion of Gorenstein flat dimension for unbounded
complexes over left GF-closed rings.\\
Over Gorenstein rings we introduce a notion of Gorenstein cohomology
for complexes; we also define a generalized Tate cohomology for
complexes over Gorenstein rings, and we show that there is a close
connection between the absolute, the Gorenstein and the generalized
Tate cohomology.
\end{abstract}

\section {introduction}\label{sec:introduction}

In 1966 Auslander introduced the notion of G-dimension of a finite
$R$-module over a commutative noetherian local ring. In 1969
Auslander and Bridger extended this notion to two sided noetherian
rings. Calling the modules of G-dimension zero Gorenstein projective
modules, in 1995 Enochs and Jenda defined Gorenstein projective
(whether finitley generated or not) and Gorenstein injective modules
over an arbitrary ring. 
Another extension of the G-dimension is based on Gorenstein flat
modules. These modules were introduced by Enochs, Jenda and
Torrecillas (\cite{enochs:93:gorenstein}).

Gorenstein homological algebra is the relative version of
homological algebra that uses the Gorenstein projective (Gorenstein
injective, Gorenstein flat respectively) modules instead of the
usual projective (injective, flat respectively) modules. The
Gorenstein dimensions for modules are defined in a similar manner
with the classical homological dimensions, but using Gorenstein
projective (Gorenstein injective, Gorenstein flat respectively)
resolutions instead of projective (injective, flat respectively)
resolutions.

It seems quite likely that there is a version of Gorenstein
homological algebra in the category of complexes. For homologically
right bounded 
complexes over commutative rings, Yassemi
(\cite{yassemi:95:gorenstein}) and Christensen
(\cite{christensen:00:gorenstein}) introduced a Gorenstein
projective dimension. Christensen, Frankild and Holm gave
generalizations of the Gorenstein projective, Gorenstein injective
and Gorenstein flat dimensions to homologically right bounded
complexes  (\cite{christensen:06:ongorenstein}). Veliche
(\cite{veliche:04:gorenstein}) extended the concept of Gorenstein
projective dimension to the setting of unbounded complexes over
associative rings. Asadollahi and Salarian
(\cite{asadollahi:06:gicomplexes}) defined the dual notion, that of
Gorenstein injective dimension of complexes over an associative
ring.

We define here a notion of Gorenstein flat dimension of unbounded
complexes over left GF-closed rings. These are the rings for which
the class of left Gorenstein flat modules is closed under
extensions.
The class of left GF-closed rings includes (strictly) the one of right coherent rings and the one of rings of finite weak dimension 
(for examples of left GF-closed rings that are neither right
coherent nor of finite weak dimension see \cite{bennis:08:rings}).

 Our definition of Gorenstein flat dimension for complexes is
given by means of DG-flat resolutions. We say that the Gorenstein
flat dimension of a complex $N$ of left $R$-modules, $Gfd N$, is
less than or equal to $g \in Z$ , if there exists a DG-flat
resolution $F \rightarrow N$ with $sup H(F) \le g$ and with $C_j(F)$
Gorenstein flat for all $j \ge g$. If $Gfd N \le g$ for all integers
$g$ then $Gfd N = - \infty$; if $Gfd N \le g$ does not hold for any
$g$ then $Gfd N = \infty$. We show that most properties of modules
of finite Gorenstein flat dimension are preserved for complexes of
finite Gorenstein flat dimension. We also show that for a right
homologically bounded complex our definition agrees with
\cite{christensen:06:ongorenstein}, Definition 2.7 .

The second part of this paper deals with Gorenstein cohomology and
generalized Tate cohomology for complexes over Gorenstein rings. The
fact that over such a ring every complex has a special Gorenstein
projective precover (\cite{garcia:99:covers}) allows us to define a
notion of Gorenstein relative cohomology for complexes: if $G
\rightarrow M$ is a special Gorenstein projective precover of $M$,
then for a complex $N$ we define the nth Gorenstein cohomology group
$Ext_{\mathcal{G}} ^n (M, N)$ by the equality $Ext_{\mathcal{G}} ^n
(M, N)= H^n \mathcal{H}om (G, N)$.

Over Gorenstein rings again, we also define a notion of generalized
Tate cohomology, $\overline{Ext} ^n (M, -)$, by the combined use of
a special Gorenstein projective precover and a DG-projective
precover of $M$. We show that there is a close connection between
the absolute, the Gorenstein and this generalized Tate cohomology:
for each complex $N$ there exists an exact sequence $ \ldots
\rightarrow  Ext^n_{\G} (M,N) \rightarrow Ext^n_R(M,N) \rightarrow
\overline{Ext}^n (M,N) \rightarrow \ldots$ (where $Ext_R ^n (-,-)$
are the absolute cohomology functors). We also prove that for a
bounded complex $M$ over a Gorenstein ring, $\overline{Ext} ^n (M,
N) \simeq \widehat{Ext}^n(M,N)$ for $n > sup M$, for any $R$-module
$N$ (where $\widehat{Ext}^n(-,-)$ are the Tate cohomology functors
introduced by Veliche \cite{veliche:04:gorenstein}).

\section{Preliminaries}\label{sec:preliminaries}
Let $R$ be an associative ring with unit. By $R$-module we mean left
$R$-module.

A (chain) complex $C$ of $R$-modules is a sequence $C = \ldots
\rightarrow C_2 \xrightarrow{\partial_2} C_1
\xrightarrow{\partial_1} C_0 \xrightarrow{\partial_0} C_{-1}
\xrightarrow{\partial_{-1}} C_{-2} \rightarrow \ldots$ of
$R$-modules and $R$-homomorphisms such that $\partial_{n-1}\circ
\partial_n = 0$ for all $n\in \mathbb{Z}$.

A complex $C$ is exact if for each n, $Ker \partial_n = Im
\partial_{n+1}$.


\begin{definition}
A module $M$ is Gorenstein projective if there is a $Hom (-, Proj)$
exact exact complex $ \ldots \rightarrow P_1 \rightarrow P_0
\rightarrow P_{-1} \rightarrow P_{-2}\rightarrow \ldots $ of
projective modules such that $M = Ker(P_0 \rightarrow P_{-1})$.
\end{definition}

The class of Gorenstein projective modules is projectively
resolving, i.e. if $0 \rightarrow M' \rightarrow M \rightarrow M"
\rightarrow 0$ is an exact sequence with $M"$ a Gorenstein
projective module, then $M'$ is Gorenstein projective if and only if
$M$ is Gorenstein projective (\cite{holm:04:gorenstein}, Theorem
2.5).

The dual notion is that of Gorenstein injective module:
\begin{definition}
A module $G$ is Gorenstein injective if there is a $Hom (Inj, -)$
exact exact complex $ \ldots \rightarrow E_1 \rightarrow E_0
\rightarrow E_{-1} \rightarrow E_{-2}\rightarrow \ldots $ of
injective modules such that $G = Ker(E_0 \rightarrow E_{-1})$.
\end{definition}

The class of Gorenstein injective modules is injectively resolving:
if $0 \rightarrow G' \rightarrow G \rightarrow G" \rightarrow 0$ is
an exact sequence with $G'$ a Gorenstein injective module, then $G"$
is Gorenstein injective if and only if $G$ is Gorenstein injective
(\cite{holm:04:gorenstein}, Theorem 2.6)

The Gorenstein flat modules are defined in terms of the tensor
product:

\begin{definition}
A module $N$ is Gorenstein flat if there is an $Inj \otimes -$ exact
exact sequence $ \ldots \rightarrow F_1 \rightarrow F_0 \rightarrow
F_{-1} \rightarrow F_{-2}\rightarrow \ldots $ of flat modules such
that $N = Ker(F_0 \rightarrow F_{-1})$.
\end{definition}

Enochs and Jenda proved (\cite{enochs:00:relative}) that over
Gorenstein rings the class of Gorenstein flat modules is
projectively resolving. Holm (\cite{holm:04:gorenstein}) showed that
if the ring $R$ is right coherent then the class of left Gorenstein
flat modules is projectively resolving.

Bennis (\cite{bennis:08:rings}) calls a ring $R$ left GF-closed if
the class of left Gorenstein flat modules is closed under
extensions. Over such a ring, the class of Gorenstein flat modules
is projectively resolving (\cite{bennis:08:rings}, Theorem 2.3).

A Gorenstein projective (flat) resolution of a module $M$ is an
exact sequence $\ldots \rightarrow P_1 \rightarrow P_0 \rightarrow M
\rightarrow 0$, with each $P_j$ Gorenstein projective (Gorenstein
flat).

\begin{definition}
The Gorenstein projective (flat) dimension of an $R$-module $M$,
$Gpd_R M$ ($Gfd_R M$ respectively) is defined as the least integer n
such that there is a Gorenstein projective (flat) resolution $0
\rightarrow P_n \rightarrow \ldots \rightarrow P_0 \rightarrow M
\rightarrow 0$ . If such an n does not exist then the Gorenstein
projective (flat) dimension of $M$ is $\infty$.
\end{definition}

If $Gpd M = n$ then any n-syzygy of $M$ is Gorenstein projective.\\
If the ring is left GF-closed and $Gfd M = n$ then any n-flat syzygy of $M$ is Gorenstein flat.\\

The Gorenstein injective dimension is defined in terms of Gorenstein
injective resolutions. A Gorenstein injective resolution of a module
$M$ is an exact complex $0 \rightarrow M \rightarrow G^0 \rightarrow
G^1 \rightarrow \ldots$ with each $G^j$ Gorenstein injective.

\begin{definition}
The Gorenstein injective dimension of a module $M$, $Gid_R M$, is
the least integer $n \ge 0$ such that there is a Gorenstein
injective resolution $0 \rightarrow M \rightarrow G^0 \rightarrow
\ldots \rightarrow G^n \rightarrow 0$. If such an n does not exist
then $Gid M = \infty$.
\end{definition}


In the category of complexes the projective (injective, flat)
dimension were defined by Avramov and Foxby, by means of
DG-projective (DG-injective, DG-flat) resolutions.

 If $X$ and $Y$ are both complexes of left
$R$-modules then $\mathcal Hom(X,Y)$ denotes the complex with
$\mathcal Hom(X,Y)_n=\Pi_{q=p+n}Hom_R(X_p,Y_q)$ and with
differential given by $\partial(f)=\partial^Y \circ
f-(-1)^nf\circ\partial^X$, for $f\in \mathcal Hom(X,Y)_n$.

\begin{definition}
A complex $P$ is DG-projective if each $P_n$ is projective and
$\mathcal{H}om (P, E)$ is exact for any exact complex $E$.
\end{definition}

For example, every right bounded complex of projective modules \\
$P = \ldots \rightarrow P_{n_0+2} \rightarrow  P_{n_0+1} \rightarrow
P_{n_0} \rightarrow 0$ is a
DG-projective complex (\cite{avramov:91:homological}, Remark 1.1 P).\\
The class of DG-projective complexes is projectively resolving: if
$0 \rightarrow P' \rightarrow P \rightarrow P" \rightarrow 0$ is a
short exact sequence of complexes with $P"$  DG-projective, then
$P'$ is DG-projective if and only if $P$ is DG-projective.
(\cite{enochs:96:orthogonality}, Remark pp. 31).

A quasi-isomorphism $P \rightarrow X$ with $P$ DG-projective is
called a DG-projective resolution of $X$. By
(\cite{enochs:96:orthogonality}, Corollary 3.10), every complex has
a surjective DG-projective resolution.

Veliche defined the Gorenstein projective dimension for unbounded
complexes over associative rings; her definition uses complete
resolutions.

\begin{definition}
Let $N$ be a complex of left $R$-modules. A complete resolution of
$N$ is a diagram $T \xrightarrow{u} P \rightarrow N$ with $P
\rightarrow N$ a DG-projective resolution, $T$ a $Hom (-, Proj)$
exact exact complex of projective modules, and $u$ a map of
complexes such that $u_i$ is bijective for $i \gg 0$.
 \end{definition}

\begin{definition}
The Gorenstein projective dimension of a complex $N$ is defined by
$Gpd N$ = $inf \{n \in Z$, $T \xrightarrow{u} P \rightarrow N $ is a
complete resolution such that $u_i$ is bijective for each $i \ge
n\}$.
\end{definition}

The dual notion of DG-projective complex is that of DG-injective
complex.
\begin{definition}
A complex $I$ is DG-injective if each $I^n$ is injective and if
$\mathcal Hom(E,I)$ is exact for any exact complex $E$.
\end{definition}
For example, every left bounded complex of injective modules \\
$I = 0 \rightarrow I_{n_0} \rightarrow I_{n_0 -1} \rightarrow
\ldots$ is
DG-injective (\cite{avramov:91:homological},Remark 1.1 I).\\
By \cite{enochs:96:orthogonality}, Remark, pp. 31, the class of
DG-injective complexes is injectively resolving.\\
A DG-injective resolution of $X$ is a quasi-isomorphism $X
\rightarrow I$ with $I$ DG-injective. By
\cite{enochs:96:orthogonality}, Corollary 3.10, every complex has an
injective DG-injective resolution.\\
The Gorenstein injective dimension for unbounded complexes over
associative rings was introduced by Asadollahi and Salarian. Their
definition is given in terms of complete coresolutions:

\begin{definition}
Let $N$ be a complex of left $R$-modules. A complete coresolution of
$N$ is a diagram $N \rightarrow I \xrightarrow{\nu} T$ with $N
\rightarrow I$ a DG-injective resolution, $T$ a $Hom (Inj, -)$ exact
exact complex of injective modules, and $\nu$ a map of complexes
such that ${\nu}_i$ is bijective for $i \ll 0$.
 \end{definition}

\begin{definition}
The Gorenstein injective dimension of a complex $N$ is defined by
$Gid N$ = $inf \{-n \in Z$, $N \rightarrow I \xrightarrow{\nu} T$ is
a complete coresolution such that ${\nu}_i$ is bijective for each $i
\le n \}$.
\end{definition}



The DG-flat complexes are defined in terms of the tensor product of
complexes. The tensor product of a complex of right $R$-modules $X$
and a complex of left $R$-modules $Y$ is the complex of $Z$-modules
$X \otimes Y$ with $(X \otimes Y)_n = \oplus _{t \in Z} X_t \otimes
Y_{n-t}$ and $\delta (x \otimes y) = \delta_t ^X (x) \otimes y +
(-1)^t x \otimes \delta_{n-t}^Y (y)$ for $x \in X_t$ and $y \in
Y_{n-t}$.

\begin{definition}
A complex $F$ of left $R$-modules is DG-flat if each $F_n$ is flat
and for any exact complex $E$ of right $R$-modules the complex $E
\otimes F$ is exact.
\end{definition}

\begin{remark}(\cite{garcia:99:covers}, page 118)
Let $E$ be an exact complex of right $R$-modules. Since $(E \otimes
F)^+ \simeq \mathcal{H}om (E, F^+)$ we have that $F$ is DG-flat if
and only if $F^+$ is DG-injective.
\end{remark}

Thus every right bounded complex of flat modules is DG-flat.

The class of DG-flat complexes is projectively resolving. 

A DG-flat resolution of a complex $X$ is a quasi-isomorphism $F
\rightarrow X$ with $F$ a DG-flat complex. Since every DG-projective
complex is DG-flat, every complex has a surjective
DG-flat resolution.\\
We also recall the definition of a projective (injective, flat
respectively) complex:

\begin{definition}(\cite{garcia:99:covers}, Theorem 3.1.3, Theorem 4.1.3)
A complex $F= \ldots \rightarrow F_{n+1} \xrightarrow{\delta}_{n+1}
F_n \xrightarrow{\delta}_n F_{n-1} \rightarrow \ldots$ is projective
(injective, flat respectively) if $F$ is exact, each $F_n$ is
projective (injective, flat respectively) and $Ker \delta_n$ is
projective (injective, flat respectively) for all $n \in Z$.
\end{definition}

By \cite{enochs:96:orthogonality}, Proposition 3.7, a complex is
projective (injective, flat respectively) if and only if it is exact
and DG-projective (DG-injective, DG-flat respectively).






\section{Gorenstein flat dimension for complexes}


We recall that a ring $R$ is left GF-closed if the class of
Gorenstein flat left $R$-modules is closed under extensions, i.e. if
$0 \rightarrow A \rightarrow B \rightarrow C \rightarrow 0$ is an
exact sequence with $A$ and $C$ Gorenstein flat modules, then $B$
is also Gorenstein flat.\\
Bennis proved (\cite{bennis:08:rings}, Theorem 2.3) that a ring $R$
is left GF-closed if and only if the class of Gorenstein flat left
$R$-modules is projectively resolving.

 Throughout this section $R$ denotes a left
GF-closed ring. 




In \cite{christensen:06:ongorenstein} the authors define the
Gorenstein flat dimension for
homologically right-bounded complexes over associative rings: \\
\begin{definition}
The Gorenstein flat dimension, $Gfd_R X$, of a homologically
right-bounded complex $X$ is defined as\\
$Gfd_R X = inf \{ sup\{l\in Z/ A_l \neq 0\}, A$ a right-bounded
complex of Gorenstein flat modules such that $A \simeq X \}$, where
$\simeq$ is the equivalence relation induced by quasi-isomorphisms.

\end{definition}

We introduce a notion of Gorenstein flat dimension for unbounded
complexes over left GF-closed rings. We show that for homologically
right-bounded complexes over left GF-closed rings, our definition
agrees with \cite{christensen:06:ongorenstein}, 2.7\\

We show first that if a complex $N$ has a DG-flat resolution $F
\rightarrow N$ with $sup H(F) \le g$ and with $C_j(F)$ Gorenstein
flat for $j \ge g$, then for every DG-flat resolution $F'
\rightarrow N$, $sup H(F') \le g$ and $C_j(F')$ is Gorenstein flat
for all $j \ge g$. (We recall that for a complex $F = \ldots
\rightarrow F_{n+1} \xrightarrow{f_{n+1}} F_n \xrightarrow{f_n}
F_{n-1} \ldots $, $C_j(F)$ denotes the module $coker f_{j+1}$.)\\

For two complexes, $X$ and $Y$, we denote by $Hom(X,Y) = Z^0
\mathcal{H}om (X, Y)$
the group of maps of complexes from $X$ to $Y$.\\
We denote by $Ext(-,-)$ the right derived functors of $Hom(-,-)$. By
\cite{enochs:96:orthogonality}, Proposition 3.5, a complex $P$ is
DG-projective if and only if $Ext^1(P, E)=0$ for any exact complex
$E$.

We begin by proving:\\

\begin{lemma}
Let $R$ be a left GF-closed ring and let $N$ be a complex of left
$R$-modules. If $N$ has a DG-flat resolution $F \rightarrow N$ such
that $sup H(F) \le g$ and $C_j(F)$ is Gorenstein flat for $j \ge g$,
then for any DG-projective resolution $P \rightarrow N$, we have
$sup H(P) \le g$
 and $C_j(P)$ is Gorenstein flat for any $j \ge g$.
\end{lemma}

\begin{proof}
If $P \rightarrow N$ is a DG-projective resolution then $sup H(P)
= sup H(N) = sup H(F) \le g$.\\
We can assume without loss of generality that $F \rightarrow N$ is a
surjective DG-flat resolution (if not, let $\overline{F} \rightarrow
N$ be surjective with $\overline{F}$ a flat complex; then $F \oplus
\overline{F} \rightarrow N$ is a surjective DG-flat resolution and $C_j(F) \oplus C_j(\overline{F})$ is Gorenstein flat for all $j \ge g$).\\
Then there is an exact sequence $0 \rightarrow U \rightarrow F
\rightarrow N \rightarrow 0$ with $U$ exact; this gives an exact
sequence $0 \rightarrow Hom(P, U) \rightarrow Hom(P,F) \rightarrow
Hom(P, N) \rightarrow Ext^1 (P, U) =0$
(\cite{enochs:96:orthogonality}, Proposition 3.6). So there exists a
map of complexes $P \rightarrow F$ that makes the diagram

\[
\begin{diagram}
\node{}\node{P}\arrow{sw,t,..}{u}\arrow{s}\\
\node{F}\arrow{e}\node{N}
\end{diagram}
\]

commutative. Since both $P \rightarrow N$ and $F \rightarrow N$ are
quasi-isomorphisms, so is $P \rightarrow F$.\\
We can assume that $P \rightarrow F$ is a surjective
quasi-isomorphism (if not, let $\overline{P} \rightarrow F$ be
surjective with $\overline{P}$ a projective complex; then $P \oplus
\overline{P} \rightarrow F$ is a surjective quasi-isomorphism). Then
there exists an exact sequence $0 \rightarrow V \rightarrow P
\rightarrow F \rightarrow 0$, with $V$ an exact complex. Both $F$
and $P$ are DG-flat complexes, so $V$ is DG-flat. Thus $V$ is exact
and DG-flat,
so $V$ is flat.         \hspace{92mm}        (1)\\
We have an exact sequence $0 \rightarrow C_g(V) \rightarrow C_g(P)
\rightarrow C_g(F) \rightarrow 0$ with $C_g(F)$ Gorenstein flat and
$C_g(V)$ flat (by (1)). It follows that $C_g(P)$ is Gorenstein flat.
Since the complex $ \ldots \rightarrow P_{g+1} \rightarrow P_g
\rightarrow C_g(P) \rightarrow 0$ is exact with $C_g(P)$ Gorenstein
flat, each $P_j$ flat and the ring is left GF-closed it follows that
$C_j(P)$ is Gorenstein flat for any $j \ge g$.\\

\end{proof}

We will also use the following result:

\begin{lemma}
Let $R$ be a left GF-closed ring. Let $0 \rightarrow A \rightarrow B
\rightarrow C \rightarrow 0$ be an exact sequence of $R$-modules. If
$A$ is flat, $B$ is Gorenstein flat and $C^+$ is Gorenstein
injective, then $C$ is Gorenstein flat.
\end{lemma}

\begin{proof}
Since $B$ is Gorenstein flat there is an exact sequence $0
\rightarrow B \rightarrow F \rightarrow Y \rightarrow 0$, with $F$
flat and $Y$ Gorenstein flat. Consider the push out
diagram:\\

\[
\begin{diagram}
\node{}\node{}\node{0}\arrow{s}\node{0}\arrow{s}\\
\node{0}\arrow{e}\node{A}\arrow{s,=}\arrow{e}\node{B}\arrow{s}\arrow{e}\node{C}\arrow{s}\arrow{e}\node{0}\\
\node{0}\arrow{e}\node{A}\arrow{e}\node{F}\arrow{s}\arrow{e}\node{X}\arrow{s}\arrow{e}\node{0}\\
\node{}\node{}\node{Y}\arrow{s}\arrow{e,=}\node{Y}\arrow{s}\\
\node{}\node{}\node{0}\node{0}
\end{diagram}
\]

The exact sequence $0 \rightarrow C \rightarrow X \rightarrow Y
\rightarrow 0$ gives an exact sequence $0 \rightarrow Y^+
\rightarrow X^+ \rightarrow C^+ \rightarrow 0$. By hypothesis, $C^+$
is Gorenstein injective. By \cite{holm:04:gorenstein}, Theorem 3.6,
$Y^+$ is Gorenstein injective. Since the class of Gorenstein
injective modules is injectively resolving it follows that $X^+$ is
Gorenstein injective. The exact sequence $0 \rightarrow A
\rightarrow F \rightarrow X \rightarrow 0$ also gives an exact
sequence $0 \rightarrow X^+ \rightarrow F^+ \rightarrow A^+
\rightarrow 0$. Since $A$ and $F$ are flat, both $F^+$ and $A^+$ are
injective. So $id X^+ \le 1$. Since $X^+$ is Gorenstein injective
and has finite injective dimension it
follows (\cite{enochs:00:relative}, Proposition 10.1.2) that $X^+$ is injective. By \cite{enochs:00:relative}, Th. 3.2.10, $X$ is flat.\\
Since $R$ is left GF-closed, the sequence $0 \rightarrow C
\rightarrow X \rightarrow Y \rightarrow 0$ is exact, $X$ is flat and
$Y$ is Gorenstein flat, it follows that $C$ is Gorenstein flat.
\end{proof}

And, we will also use:

\begin{lemma}
Let $R$ be a left GF-closed ring, and let $N$ be a complex of left
$R$-modules. If $N$ has a DG-flat resolution $F \rightarrow N$ with
$sup H(F) \le g$ and with $C_j(F)$ Gorenstein flat for $j \ge g$,
then $Gid N^+ \le g$.\end{lemma}

\begin{proof}
Since $F \rightarrow N$ is a DG-flat resolution, it follows that
$N^+ \rightarrow F^+$ is a DG-injective resolution.\\
We have $sup H(F) \le g$, so $inf H(F^+) \ge -g$.\\
Let $C$ denote $C_g(F)$. The exact sequence $ \ldots \rightarrow
F_{g+1} \xrightarrow{g_{g+1}} F_g \xrightarrow{\pi} C \rightarrow 0$
gives an exact sequence $0 \rightarrow C^+ \xrightarrow{\pi ^+}
F^+_g \xrightarrow{f^+_{g+1}} F^+_{g+1} \rightarrow \ldots$.\\
 So $ker f^+_{g+1} = Im \pi^+ \simeq C^+$ is Gorenstein injective
(by \cite{holm:04:gorenstein}, Theorem 3.6, since $C$ is Gorenstein
flat). Then there exists an exact sequence $\overline{T} = \ldots
T_{g+2} \rightarrow T_{g+1} \rightarrow Im \pi^+ \rightarrow 0$,
with each $T_j$ injective and such that the sequence is $Hom (Inj,
-)$ exact. Since each $F^+_j$ is injective we have a commutative
diagram

\[
\begin{diagram}
\node{\cdots}\arrow{e}\node{F_{g-2}^+}\arrow{s,r}{u_g}\arrow{e}
\node{F_{g-1}^+}\arrow{s,r}{u_{g-1}}\arrow{e,t}{\pi^+}\node{Im\;\pi^+}
\arrow{s,=}\arrow{e}\node{0}\\
\node{\cdots}\arrow{e}\node{T_{-g+2}}\arrow{e}\node{T_{-g+1}}\arrow{e}
\node{Ker\;f_{g+1}^+}\arrow{e}\node{0}
\end{diagram}
\]

Since $ \ldots \rightarrow F_{g+1} \rightarrow F_g \rightarrow C
\rightarrow 0$ is exact and $Inj \otimes -$ exact, it follows that
$0 \rightarrow Ker f^+_{g+1} \rightarrow F^+_g \rightarrow F^+_{g+1}
\rightarrow \ldots$ is an exact complex of right $R$-modules that is
also $Hom (Inj, -)$ exact.

So we have a commutative diagram

\[
\begin{diagram}
\node{F^+=\cdots}\arrow{e}\node{F_{g-1}^+}\arrow{s,r}{u_{g-1}}\arrow{e}
\node{F_{g}^+}\arrow{s,=}\arrow{e}\node{F_{g+1}^+}
\arrow{s,=}\arrow{e}\node{\ldots}\\
\node{T=\cdots}\arrow{e}\node{T_{-g+1}}\arrow{e}\node{F_{g}^+}\arrow{e}
\node{F_{g+1}^+}\arrow{e}\node{\ldots}
\end{diagram}
\]

with $T$ an exact complex of injective modules which is also $Hom
(Inj, -)$ exact.\\
 Since $N^+$ has a complete coresolution $N^+ \rightarrow F^+
 \xrightarrow{u} T$ with $u_j = 1_{F^+_j}$ bijective for all $j \le
 -g$ it follows that $Gid N^+ \le g$ (\cite{asadollahi:06:gicomplexes}, Theorem 2.3).

\end{proof}

We can prove now:

\begin{lemma}
Let $R$ be a left GF-closed ring. If a complex $N$ has a DG-flat
resolution $F \rightarrow N$ such that $sup H(F) \le g $ and
$C_j(F)$ is Gorenstein flat for all $j \ge g$, then for any DG-flat
resolution $F' \rightarrow N$ we have that $sup H(F') \le g$ and
$C_i (F')$ is Gorenstein flat for all $i \ge g$.
\end{lemma}

\begin{proof}

Let $F' \rightarrow N$ be another DG-flat resolution. Then $sup
H(F') = sup H(N) = supH(F) \le g$.

 Since $F' \rightarrow N$ is a DG-flat resolution it follows that
 $N^+ \rightarrow F'^+$ is a DG-injective resolution.\\
  By Lemma 3, $Gid N^+ \le g$. By \cite{asadollahi:06:gicomplexes},
  Theorem 2.3, there exists an injective complete coresolution $N
  \rightarrow F'^+ \xrightarrow{u} T$ with $u_j$ bijective for all
  $j \le -g$. Then $Z_j(F'^+) \simeq Z_j(T)$ is Gorenstein injective
  for all $j \le -g$, that is $Ker f'^+_j$ is Gorenstein injective for all $j \ge
  g+1$.

The exact sequence $F'_{g+1} \xrightarrow{f_{g+1}} F'_g
\xrightarrow{\pi} C_g(F') \rightarrow 0$ gives an exact sequence $0
\rightarrow C_g(F')^+ \xrightarrow{\pi^+} {F'_g}^+
\xrightarrow{{f'^+
_{g+1}}} {F'_{g+1}}^+$. So $C_g(F')^+ \simeq ker {f'^+ _{g+1}}$ is Gorenstein injective.\\ 
Let $P \rightarrow F'$ be a surjective DG-projective resolution.
Then $P \rightarrow N$ is a DG-projective resolution, so by Lemma 1,
$C_j
(P)$ is Gorenstein flat for all $j \ge g$.\\
 There is an exact sequence $0 \rightarrow V \rightarrow P
 \rightarrow F' \rightarrow 0$ with $V$ exact. Since $V$ is also DG-flat (because $P$ and $F'$
 are DG-flat), it follows that $V$ is
 flat. Then $C_j(V)$ is flat, for all j.\\
 We have an exact sequence $0 \rightarrow C_g(V) \rightarrow C_g(P)
 \rightarrow C_g(F') \rightarrow 0$ with $C_g(V)$ flat and $C_g(P)$
 Gorenstein flat, and with $C_g(F')^+$ Gorenstein injective. By Lemma 2,
 $C_g(F')$ is Gorenstein flat.

 Since $H_i(F') =0$ for all $i \ge
 g+1$ and each $F'_n$ is flat it follows that $C_i(F')$ is
 Gorenstein flat for all $i \ge g$.

\end{proof}

\begin{definition}
Let $R$ be a left closed GF-ring. Let $N$ be a complex of left
$R$-modules. The Gorenstein flat dimension of $N$ is defined by:\\
$Gfd N \le g$ if there is a DG-flat resolution $F \rightarrow N$
such that $sup H(F) \le g$ and $C_j(F)$ is Gorenstein flat for any
$j \ge g$. If $Gfd N \le g$ but $Gfd N \le {g-1}$ does not hold then
$Gfd N
=g$.\\
If $Gfd N \le g$ for any $g$ then $Gfd N = - \infty$.\\
If $Gfd N \le g$ does not hold for any $g$ then $Gfd N = \infty$.\\

\end{definition}

\begin{remark}
$Gfd N = - \infty$ if and only if $N$ is an exact complex.
\end{remark}

\begin{proof}
"$\Rightarrow$" If $Gfd N = - \infty$ then $H_i (N) =0$ for any
integer $i$.\\
"$\Leftarrow$" If $N$ is exact and $F \rightarrow N$ is a surjective
DG-flat resolution then $F$ is a flat complex. Then $H_j(F)=0$ and
$C_j(F)$ is flat hence Gorenstein flat for any integer j. So $Gfd N
= - \infty$.
\end{proof}

\begin{theorem}
Let $N$ be a complex of $R$-modules. The following are equivalent:\\
1) $Gfd N \le g$;\\
2) $sup H(N) \le g$ and $C_i (F)$ is Gorenstein flat for any $i \ge
g$,
for any DG-flat resolution $F \rightarrow N$;\\
3) For any DG-projective resolution $P \rightarrow N$ we have $sup
H(P) \le g$ and $C_j(P)$ is Gorenstein flat for any $j
\ge g$;\\
4) There exists a DG-projective resolution $P \rightarrow N$ such
that $H_j(P) = 0$ for any $j \ge g+1$, and $C_j(P)$ is Gorenstein
flat for any $j \ge g$.\\
\end{theorem}

\begin{proof}
1) $\Rightarrow$ 2) by Lemma 4;\\
2) $\Rightarrow$ 1) straightforward;\\
1) $\Rightarrow$ 3)  By Lemma 1; \\
3) $\Rightarrow$ 4) Straightforward;\\
4) $\Rightarrow$ 1) By definition, since every DG-projective
resolution is a DG-flat resolution.

\end{proof}

\textbf{Properties of dimensions}

We recall that the flat dimension of a complex $N$, $fd N$, is
defined (\cite{avramov:91:homological}) by $fd N \le g$ if $sup H(N)
\le g$ and for any DG-flat complex $F$, such that $F \simeq N$,
$C_j(F)$ is flat for all $j \ge g$ (where
$\simeq$ is the equivalence relation generated by quasi-isomorphisms).\\
If $fd N \le g$ does not hold for any $g \in Z$, then $fd N =
\infty$; if $fd N \le g$ holds for any $g \in Z$ then $fd N = -
\infty$.\\

\begin{proposition}
 Let $R$ be a left GF-closed ring. For any complex of $R$-modules $N$, we have $Gfd N \le fd N$ with
equality if $fd N < \infty$.
\end{proposition}

\begin{proof}
- Clear if $fd N = \infty$\\
- If $fd N = - \infty$ then $N$ is exact, so $Gfd N = - \infty$\\
- Let $fd N =g < \infty$. Then for any DG-flat resolution $F
\rightarrow N$ we have $sup H(F) \le g$ and $C_j(F)$ is flat, hence
Gorenstein flat, for all $j \ge g$. By definition, $Gfd N \le
g$.\\
Suppose $Gfd N \le g-1$. Then for any DG-flat resolution $F
\rightarrow N$, $C_j(F)$ is Gorenstein flat for all $j \ge
g-1$, and $H_j(F) = 0$ for all $j \ge g$. \\
The exact sequence $0 \rightarrow C_g(F) \rightarrow F_{g-1}
\rightarrow C_{g-1}(F) \rightarrow 0$ with $C_g(F)$ and $F_{g-1}$
flat modules gives that $C_{g-1}(F)$ has finite flat dimension.
Since $C_{g-1}(F)$ is a Gorenstein flat module of finite
flat dimension it follows that $C_{g-1}(F)$ is flat (\cite{enochs:00:relative}, Corollary 10.3.4).\\
But then $fd N \le g-1$. Contradiction.\\
So $Gfd N = fd N$ if $fd N < \infty$.

\end{proof}

\begin{proposition}
Let $R$ be a left GF-closed ring. Let $M$ be an $R$-module. If
$\overline{M}$ is $M$ as a complex at zero then $Gfd \overline{M} =
Gfd_R M$.
\end{proposition}

\begin{proof}

Let $ \ldots \rightarrow F_2 \rightarrow F_1 \rightarrow F_0
\rightarrow M \rightarrow 0$ be a flat resolution of $M$. Then $F
\rightarrow \overline{M} $ is a DG-flat resolution (where $F=\ldots
 \rightarrow F_1
\rightarrow F_0 \rightarrow 0$).\\
- Case $Gfd_R M = \infty$.\\
Suppose $Gfd \overline{M}=l < \infty$. Then $C_j(F)$ is Gorenstein
flat for any $j \ge l$. Since $0 \rightarrow C_l(F) \rightarrow
F_{l-1} \rightarrow \ldots \rightarrow F_0 \rightarrow M \rightarrow
0$ is exact with $C_l(F)$ Gorenstein flat and each $F_j$ flat, it
follows that $Gfd_R M \le l$. Contradiction. So $Gfd \overline{M} =
\infty$.
- Case $Gfd_R M =l < \infty$.\\
Then $C_l(F)$ is Gorenstein flat; since the ring is GF-closed,
$C_j(F)$ is Gorenstein flat for any $j \ge l$. Then $F \rightarrow
\overline{M}$ is a DG-flat resolution with $C_j (F)$ Gorenstein flat
for all $j \ge l$ and with
$H_j(F) = 0$ for any $j \ge 1$. By definition, $Gfd \overline{M} \le l$. \\
Suppose $Gfd \overline{M} \le l-1$. Then $C_{l-1}(F)$ is Gorenstein
flat. The exact sequence $ 0 \rightarrow C_{l-1} \rightarrow F_{l-2}
\rightarrow \ldots \rightarrow F_0 \rightarrow M \rightarrow 0$ with
$C_{l-1}$ Gorenstein flat and each $F_j$ flat gives that $Gfd_R (M)
\le l-1$. Contradiction. So $Gfd \overline{M} = l$.

\end{proof}

\begin{proposition}
Let $R$ be a left GF-closed ring and let $N$ be a complex of left $R$-modules.\\
a) If $Gfd N \le g$ then $Gid N^+ \le g$.\\
b) If $R$ is right coherent then $Gfd N \le g$
if and only if $Gid N^+ \le g$.\\
c) If $R$ is right coherent then $Gfd N = Gid N^+$.
\end{proposition}

\begin{proof}


 a) By Lemma 3.

b) Let $R$ be right coherent. Let $N$ be a complex of left
$R$-modules with $Gid N^+ \le g$.

 If $F \rightarrow N$ is a DG-flat
resolution
then $N^+ \rightarrow F^+$ is a DG-injective resolution. 
Since $Gid N^+ \le g$ it follows that $Z_j(F^+)$ is a Gorenstein
injective
module for any $j \le -g$, that is $ker {f_j}^+$ is Gorenstein injective for all $j \ge g+1$, and $H_j(F^+)=0$ for all $j \le -g-1$ .\\
 The sequence $F_{g+1} \xrightarrow {f_{g+1}} F_g \xrightarrow{\pi}
 C_g(F)
\rightarrow 0$ is exact, 
therefore\\  $0 \rightarrow {C_g(F)^+} \xrightarrow{\pi^+} {F_g^+}
\xrightarrow{f_{g+1}^+} {F_{g+1} ^+}$ is exact. Then $C_g(F)^+
\simeq
 Ker {f_{g+1}^+}$ is Gorenstein injective.
Since $R$ is right coherent it follows that $C_g(F)$ is Gorenstein flat (\cite{holm:04:gorenstein}, Theorem 3.6).\\
We have $H_j(F^+) =0$ for any $j \le -g-1$. Then ${H_j(F)^+} \simeq
H_j(F^+)=0$, so $H_j(F)=0$,
for any $j \ge g+1$.\\
 Since $ \ldots \rightarrow F_{g+1}
\rightarrow F_{g+1} \rightarrow F_g \rightarrow C_g(F) \rightarrow
0$ is exact, each $F_j$ is flat, $C_g(F)$ is Gorenstein flat and the
class of Gorenstein flat modules is closed under kernels of
epimorphisms, it follows that $C_j (F)$ is Gorenstein flat for all
$j \ge g$. Thus, $Gfd N \le g$.

c) 
 If $Gfd N = - \infty$ then $N$ is exact and therefore $N^+$ is
exact, so $Gid N^+ = - \infty$;\\
 -If $Gfd N = g < \infty$ then by the above $Gid N^+ \le g$.
 Suppose $Gid N^+ \le g-1$. Since $R$ is right coherent it follows
 that $Gfd N \le g-1$. Contradiction.
 So $Gid N^+ = g$.\\
 -Case $Gfd N = \infty$\\
 Suppose that $Gid N^+ =g < \infty$. Then by b), $Gfd N \le g < \infty$. Contradiction.


\end{proof}

\begin{proposition}
Let $R$ be a left GF-closed ring. Let $0 \rightarrow N' \rightarrow
N \rightarrow N" \rightarrow 0$ be an exact sequence of complexes of
$R$-modules. If two complexes have finite Gorenstein flat dimension
then so does the third.

\end{proposition}

\begin{proof}
By \cite{veliche:04:gorenstein}, Proposition 1.3.8, there is an
exact sequence of complexes $0 \rightarrow F' \rightarrow F
\rightarrow F" \rightarrow 0$ with $F' \rightarrow N'$, $F
\rightarrow N$, and $F" \rightarrow N"$ DG-projective resolutions.
If two of the complexes $N'$, $N$, $N"$ have finite Gorenstein flat
dimension then there is $g \in Z$ such that $H_j(F')= H_j(F) =
H_j(F") = 0$ for all $j \ge g$.\\
 For each $j \ge g$ we have an
exact sequence $0 \rightarrow C_j(F') \rightarrow C_j(F) \rightarrow
C_j(F") \rightarrow 0$. If $C_j(F")$ is Gorenstein flat, then
$C_j(F')$ is Gorenstein flat if and only if $C_j(F)$ is Gorenstein
flat. If both $C_j(F')$ and $C_j(F)$ are Gorenstein flat then $Gfd_R
C_j(F") \le 1$.

\end{proof}

We recall that the finitistic projective dimension of a ring $R$,
$FPD(R)$, is defined as  $FPD(R) = sup \{pd M: pd M < \infty \}$.

We also recall \cite{holm:04:gorenstein}, Proposition 3.4:
\begin{proposition}
If $R$ is right coherent with finite left finitistic projective
dimension, then every Gorenstein projective (left) $R$-module is
also Gorenstein flat.
\end{proposition}

\begin{proposition}
 Let $R$ be right coherent of finite left finitistic projective dimension. For any complex of left $R$-modules $N$, $Gfd
N \le Gpd N$.
\end{proposition}

\begin{proof}
- Obvious if $Gpd N = \infty$;\\
- If $Gpd N = g < \infty$ then for any DG-projective resolution $P
\rightarrow N$,  $C_j (P)$ is Gorenstein projective for all $j \ge
g$ and sup $H(P) \le g$.  By Proposition 5, $C_j(P)$ is
Gorenstein flat. By Theorem 1, $Gfd N \le g$.\\
- If $Gpd N = - \infty$ then $N$ is exact, so $Gfd N = - \infty$.
\end{proof}

\begin{proposition}
Let $R$ be a left noetherian ring of finite Krull dimension and let
$N$ be a complex of $R$-modules. If $Gfd N < \infty$ then $Gpd N <
\infty$.
\end{proposition}

\begin{proof}
Let $Gfd N = g < \infty$. If $P \rightarrow N$ is a DG-projective
resolution, then $C_g(P)$ is Gorenstein flat and $H_j(P) = 0$ for
all $j \ge g+1$. By \cite{christensen:08:beyound}, Theorem 29, $Gpd
C_g(P) = l < \infty$. Then $C_{g+l} (P)$ is Gorenstein projective.
By \cite{veliche:04:gorenstein}, Theorem 3.4, $Gpd N \le g+l$.
\end{proof}

\vspace{5mm}

 For homologically bounded complexes Christensen, Frankild and
Holm
defined the Gorenstein flat dimension by \\

\begin{definition} (\cite{christensen:06:ongorenstein}, 2.7)
Let $X$ be a homologically right-bounded complex. The Gorenstein flat dimension of $X$, $Gfd_R X$, is\\
$Gfd_R X = inf \{sup$ \{$l \in Z, A_l \neq 0\}$ ,
 $A$ a right bounded complex of
Gorenstein flat modules that is isomorphic to $X$ in
$D(R)$\}\\(where $D(R)$ is the derived category) 
\end{definition}



\vspace{7mm}

\begin{remark}
Let $R$ be a left GF-closed ring, and let $X$ be a homologically
right-bounded complex. Then $Gfd_R X = Gfd X$.

\end{remark}

 \begin{proof}
 We can assume
 $inf H(X) =0$.\\
 By \cite{veliche:04:gorenstein}, 1.3.4, $X$ has a
 DG-projective resolution $F \rightarrow X$ with $inf F = 0$.\\
 If $Gfd X
\le g$ then $C_j(F)$ is Gorenstein flat for all $j \ge g$.\\
Let $\overline{F} = 0 \rightarrow C_g(F)  \rightarrow F_{g-1}
\rightarrow \ldots \rightarrow F_0 \rightarrow 0 $ and let $X' = 0
\rightarrow C_g(X) \rightarrow X_{g-1} \rightarrow X_{g-2}
\rightarrow \ldots $. The quasi-isomorphism $F \rightarrow X$ gives
a quasi-isomorphism $\overline{F} \rightarrow X'$. Since
$\overline{F} \simeq X'$ and $X \simeq X'$ in $D(R)$, it follows
that $\overline{F} \simeq X$ in $D(R)$. Each component of
$\overline{F}$
is a Gorenstein flat module, so $Gfd_R X \le g$.\\


Let $X$ be a homologically right bounded complex (with $inf
H(X)=0$), such that $Gfd_R X \le
g$ where $g < \infty$. 
We show that $Gfd X \le g$.\\ 
Since $Gfd_R X \le g$ there exists
a complex $A= 0 \rightarrow A_g \rightarrow A_{g-1} \rightarrow
\ldots \rightarrow A_0 \rightarrow
0$ of Gorenstein flat modules such that $A \simeq X$.\\ 
 $X$ is right bounded, so $X$ has a DG-projective resolution $P
\rightarrow X$ with $P$ right bounded.\\
Since $P \simeq X \simeq A$ and $P$ is DG-projective it follows that
there is a quasi-isomorphism $P \rightarrow A$ (\cite{avramov:91:homological}, 1.4.P).\\
$A$ is right bounded, so there is a surjective map $F \rightarrow A$
with $F$ a right bounded projective complex. Then $P \oplus F
\rightarrow A$ is a surjective quasi-isomorphism, so we have an
exact sequence $0 \rightarrow V \rightarrow P \oplus F \rightarrow A
\rightarrow 0$ with $V$ exact.
 Since $P_j \oplus F_j$ and $A_j$ are
Gorenstein flat modules for all $j$, it follows that each $V_j$ is
Gorenstein flat. We have $V \subseteq P \oplus F$, so $V$ is a right
bounded complex of Gorenstein flat
modules; then $C_j(V)$ is Gorenstein flat for each j.\\
The exact sequence $0 \rightarrow C_g(V) \rightarrow C_g(P) \oplus
C_g(F) \rightarrow C_g(A) \rightarrow 0$ with both $C_g(A)=A_g$ and
$C_g(V)$ Gorenstein flat gives that $C_g(P) \oplus C_g(F)$ is
Gorenstein flat. By \cite{bennis:08:rings}, Corollary 2.6, the class
of Gorenstein flat modules over a left GF-closed ring is closed
under summands, so $C_g(P)$ is Gorenstein flat. Since $P \rightarrow
X$ is a DG-projective resolution with $sup H(P) \le g$ and $C_j(P)$
Gorenstein flat
for $j \ge g$ it follows that $Gfd X \le g$. \\
 By the above, $Gfd_R X = Gfd X$ for $Gfd_R X < \infty$. Also
by the
above, $Gfd_R X = \infty$ if and only if $ Gfd X =\infty$.\\
We have $Gfd_R X = - \infty$ if and only if $X$ is exact if and only
if $Gfd X = - \infty$.
\end{proof}

We recall (\cite{enochs:00:relative}, Definition 9.1.1) that a ring
is Gorenstein if it is left and right
noetherian and has finite self injective dimension on both sides.\\
A Gorenstein ring with $id _R R$ at most n is called n-Gorenstein.
In this case $id R_R$ is also at most n (\cite{enochs:00:relative},
Proposition 9.1.8)

Gorenstein rings can be characterized in terms of Gorenstein flat
dimensions of complexes of their modules.

\begin{theorem}
Let $R$ be a left and right noetherian ring. The following are
equivalent:\\
a) $R$ is n-Gorenstein;\\
b) For every complex of $R$-modules $N$, $Gfd N \le n + sup H(N)$.
\end{theorem}

\begin{proof}
a) $\Rightarrow $ b) True if $sup H(N) = \infty$.\\
- Case $sup H(N) = l < \infty$\\
Let $F \rightarrow N$ be a DG-flat resolution. Then $H_j(F) = 0$ for
any $j>l$. So we have an exact complex $ \ldots \rightarrow F_{l+1}
\rightarrow F_l \rightarrow C_l(F) \rightarrow 0$. Since $R$ is
n-Gorenstein, $Gfd_R C_l(F) \le n$. Thus $C_j (F)$ is Gorenstein
flat for any $j \ge n+l$. Therefore $Gfd N \le n+l$.\\
b) $\Rightarrow$ a)
If $N$ is a left $R$-module then by Proposition 2, $Gfd_R N = Gfd
\overline{N} \le n$ (where $\overline{N}$ is $N$ as a complex at
zero). So every $n$-flat syzygy of $N$ is Gorenstein flat.
Similarly, every n-flat syzygy of any right $R$-module is Gorenstein
flat. It follows that $R$ is n-Gorenstein
(\cite{enochs:00:relative}, Theorem 12.3.1).

\end{proof}

\section{Gorenstein cohomology for complexes; generalized Tate cohomology for complexes}
We define the Gorenstein cohomology for complexes over Gorenstein
rings. We also define a notion of generalized Tate cohomology for
complexes over Gorenstein rings and we show that there is a close
connection between the absolute, the Gorenstein, and the generalized
Tate cohomology.

Our definition of Gorenstein cohomology for complexes uses
Gorenstein projective precovers.

We recall first that a complex $G$ is Gorenstein projective if there
exists an exact resolution of complexes:

$ \ldots \rightarrow P_2 \rightarrow P_1 \rightarrow P_0 \rightarrow
P_{-1} \rightarrow P_{-2} \rightarrow \ldots$\\
such that each $P_i$ is a projective complex, the sequence remains
exact when applying $ Hom (-, P)$, for any projective complex $P$,
and $G = Ker (P_0 \rightarrow P_{-1})$.

We recall that $Hom(X,Y) = Z^0 \mathcal{H}om (X, Y)$ is the group of
maps of complexes from $X$ to $Y$, and $Ext(-,-)$ are the right
derived functors of $Hom(-,-)$.

We also recall the definition of a Gorenstein projective precover:\\

\begin{definition}(\cite{garcia:99:covers}, Definition 1.2.3)
Let $M$ be a complex. A Gorenstein projective precover of $M$ is a
map of complexes $\phi:G \rightarrow M$ with $G$ Gorenstein
projective and with the property that for every Gorenstein
projective complex $G'$ the sequence $Hom (G', G) \rightarrow Hom
(G', M) \rightarrow
0$ is exact.\\

\end{definition}


Throughout this section we work with the projective dimension
defined by Garc\'{i}a-Rozas in \cite{garcia:99:covers}:\\
The projective dimension of a complex $M$ is the least integer $n
\ge 0$ such that there exists an exact sequence $0 \rightarrow P_n
\rightarrow P_{n-1} \rightarrow \ldots \rightarrow P_0 \rightarrow M
\rightarrow 0$ with each $P_j$ a projective complex; if such an $n$
does not exist, then the projective dimension of $M$ is $\infty$.

It is shown in \cite{garcia:99:covers} that a complex $L= \ldots
\rightarrow L_{n+1} \rightarrow L_n \rightarrow L_{n-1} \rightarrow
\ldots$ has finite projective dimension if and only if $L$ is exact,
and for each $n \in Z$, $L_n$ and $Ker (L_n \rightarrow L_{n-1})$
are modules of finite projective dimension. The class of complexes
of finite projective dimension is denoted $\L$.

Over a Gorenstein ring, Garc\'{i}a-Rozas gave the following
characterization of Gorenstein projective complexes
(\cite{garcia:99:covers}, Theorem
3.3.5):\\
\begin{theorem}
 Let $R$ be a Gorenstein ring. The following conditions are
 equivalent for a complex $G$:\\
 1) $G$ is Gorenstein
projective;\\
2) $Ext^1 (G, L) = 0$ for all complexes $L$ of finite projective
dimension;\\
3) Each $G_n$ is a Gorenstein
projective module.\\
\end{theorem}

Theorem 3 gives the following result:\\
\begin{proposition}
 The projective dimension of a Gorenstein projective complex is
either zero or infinite.
\end{proposition}

\begin{proof}
Let $G$ be a Gorenstein projective complex. Suppose $pd G = l <
\infty$ and let $0 \rightarrow P_l \xrightarrow{f_l} P_{l-1}
\rightarrow \ldots \rightarrow P_0 \xrightarrow{f_0} G \rightarrow0$
be a finite projective resolution. Since the class of Gorenstein
projectives is projectively resolving, $Im f_j$ is Gorenstein
projective for all j. By Theorem 3 (part 2), the resolution is split
exact since each $P_j$ is projective. Thus $G$ is projective.
\end{proof}

Garc\'{i}a-Rozas showed that over a Gorenstein ring every complex
$M$ has a special Gorenstein precover, i.e. a Gorenstein projective
precover $G \rightarrow M$ with $Ker (G \rightarrow M)$ a complex of
finite projective dimension.

\begin{remark}
A special Gorenstein projective precover is unique up to homotopy.
\end{remark}

\begin{proof}
Let $\phi:G \rightarrow M$ and $\phi':G' \rightarrow M$ be two
special Gorenstein projective precovers. Let $u: G \rightarrow G'$
and $v: G \rightarrow G'$ be maps of complexes induced by $1_M$.
Then $\phi' u = \phi$ and $\phi' v = \phi$. 
So $u-v: G \rightarrow Ker \phi'$. By hypothesis, $L = Ker \phi'$ is
a complex of finite projective dimension.\\ If $\theta: A
\rightarrow L$ is a special Gorenstein projective precover then we
have an exact sequence $0 \rightarrow Ker \theta \rightarrow A
\xrightarrow{\theta} L \rightarrow 0$ with $L$ and $Ker \theta$ of
finite projective dimension. Then $A$ is Gorenstein projective of
finite projective dimension, hence projective. Since $G$ is
Gorenstein projective, and $u-v: G \rightarrow L$, there
exists $\omega: G \rightarrow A$ such that $\theta \omega = u-v$.\\
By the definition of Gorenstein projective complexes, there is an
exact sequence $0 \rightarrow G \xrightarrow{j} P \rightarrow X
\rightarrow 0$ with $P$ projective and with $X$ Gorenstein
projective. This gives an exact sequence $0 \rightarrow Hom(X,A)
\rightarrow Hom(P,A)
\rightarrow Hom(G,A) \rightarrow Ext^1(X,A)=0$ (by Theorem 3).\\
 Since  $Hom(P,A)
\rightarrow Hom(G,A) \rightarrow 0$ is exact and $\omega \in
Hom(G,A)$ there exists $l:P \rightarrow A$ such that $\omega =l j$.
But $P$ is a projective complex, so $l$ is homotopic to 0. Then
$\omega =l j$ is homotopic to zero and therefore $u-v = \theta
\omega$ is homotopic to zero.\\
The identity map $1_M$ also induces a map of complexes $ \alpha: G'
\rightarrow G$. Then $ \alpha u: G \rightarrow G$ and $1_G$ are both
induced by $1_M$. By the above $ \alpha u$ and $1_G$ are homotopic.
Similarly, $u \alpha \sim 1_{G'}$.
\end{proof}

Since over a Gorenstein ring every complex has a special Gorenstein
projective precover and such a precover is unique up to homotopy, we
can compute right derived functors of $\mathcal{H}om(-,-)$ by means
of special Gorenstein projective precovers.

\begin{definition}
Let $R$ be a Gorenstein ring and let $M$ be a complex of
$R$-modules. For each complex $N$, the nth relative Gorenstein
cohomology group $Ext_{\mathcal{G}}^n (M,N)$ is defined by the
equality $Ext^n_{\mathcal{G}} (M,N) = H^n \mathcal{H}om (G, N)$,
where $G \rightarrow M$ is a special Gorenstein projective precover
of $M$.
\end{definition}

\begin{remark}
If $M$ and $N$ are modules regarded as complexes at zero, then
$Ext_{\mathcal{G}}^n (M,N)$ are the usual Gorenstein cohomology
groups.
\end{remark}

\begin{proof}
Let $\overline{M} = 0 \rightarrow M \rightarrow 0$. Let $ \ldots
\rightarrow G_1 \xrightarrow{g_1} G_0 \xrightarrow{g_0} M
\rightarrow 0$ be a special Gorenstein projective resolution of $M$
(i.e. $G_0 \xrightarrow{g_0} M$ and $G_i \xrightarrow{g_i}
Ker(G_{i-1} \rightarrow G_{i-2})$ are Gorenstein projective
precovers such that $Ker g_i$ has finite projective dimension). Then
$(\overline{G} \rightarrow \overline{M})$ is a special Gorenstein
projective precover, where $\overline{G} = \ldots \rightarrow G_1
\rightarrow G_0 \rightarrow 0$. Since $\overline{N}$ is the module
$N$ at zero, $\mathcal{H}om (\overline{G}, \overline{N})$ is the
complex $ \ldots \rightarrow Hom (G_1, N) \rightarrow Hom (G_0, N)
\rightarrow 0$. Thus $Ext_{\mathcal{G}}^n (M,N)=H^n
\mathcal{H}om(\overline{G}, \overline{N})$ are the usual Gorenstein
cohomology groups.
\end{proof}

Over a Gorenstein ring we also define generalized Tate cohomology
groups $\overline{Ext}^n (M,N)$, by the combined use of a DG-projective resolution and a special Gorenstein projective resolution of $M$.\\

Let $R$ be a Gorenstein ring, and let $M$ be a complex of
$R$-modules. Let $P \xrightarrow{\delta} M$ be a surjective
DG-projective resolution and let $G \xrightarrow{\phi} M$ be a
Gorenstein projective precover. Since $P$ is Gorenstein projective
there is a map of complexes $u: P \rightarrow G$ such that $\delta =
\phi \circ u$.\\

\begin{remark}
If $P \xrightarrow{\delta} M$ is a surjective DG-projective
precover, $G \xrightarrow{\varphi} M$ a special Gorenstein
projective precover, and $\alpha$, $\beta: P \rightarrow G$ are maps
of complexes induced by $1_M$ then $\alpha$ and $\beta$
are homotopic.\\
\end{remark}

\begin{proof}
Let $P' \xrightarrow{\theta} G$ be a surjective DG-projective
resolution. The exact sequence $0 \rightarrow E \rightarrow P'
\rightarrow G \rightarrow 0$ with $E$ exact complex gives an exact
sequence $0 \rightarrow Hom(P,E) \rightarrow Hom(P, P') \rightarrow
Hom(P, G) \rightarrow Ext^1(P,E) = 0$ (by
\cite{enochs:96:orthogonality}, Proposition 3.5). Since $Hom(P,P')
\rightarrow Hom(P, G)$ is surjective, there is a map of
complexes $u: P \rightarrow P'$ such that $\alpha = \theta u$. \\
Similarly, $\beta = \theta v$ for some $v:P \rightarrow P'$.\\
 Since
both $P \xrightarrow{\delta} M$ and $P' \xrightarrow{\varphi\theta}
M$ are DG-projective resolutions, and $u,v:P \rightarrow P'$ are
induced by $1_M$, we have $u \sim v$.
Then $\alpha = \theta u \sim \theta v = \beta$.\\
\end{proof}

\begin{definition}(generalized Tate cohomology)
Let $R$ be a Gorenstein ring and let $M$ be a complex of
$R$-modules. Let $P \rightarrow M$ be a surjective DG-projective
resolution, let $G \rightarrow M$ be a special Gorenstein projective
resolution, and let $u:P \rightarrow G$ be a map of complexes
induced by $1_M$. Let $M(u)$ be the mapping cone of $u$. For each
complex $N$ the nth generalized Tate cohomology group is defined by
the equality
$\overline Ext^n(M,N) = H^{n+1} \mathcal{H}om (M,N)$.\\
\end{definition}



We show first that the $\overline {Ext}^n(M,N)$ are well
defined.\\



- If $P \rightarrow M$ is a surjective DG-projective resolution, $G
\rightarrow M$ a special Gorenstein projective precover and $u, v :
P \rightarrow G$ are maps of complexes induced by $1_M$ then their mapping cones $M(u)$ and $M(v)$ are isomorphic.\\
 Proof: Let $P = \ldots \rightarrow P_{n+1} \xrightarrow{f_{n+1}} P_n \xrightarrow{f_n}
P_{n-1} \rightarrow \ldots$, and $G=\ldots \rightarrow G_{n+1}
\xrightarrow{g_{n+1}} G_n \xrightarrow{g_n} G_{n-1} \rightarrow
\ldots$.\\
 Since $u$ and $v$ are both induced by
$1_M$ they are homotopic; so for each $n$ there is $s_n \in Hom
(P_n, G_{n+1})$ such that $u_n -
v_n = g_{n+1} s_n + s_{n-1} f_n$.\\
There are maps of complexes $\omega:M(u) \rightarrow M(v)$, with
$\omega_n :G_{n+1} \oplus P_n \rightarrow G_{n+1} \oplus P_n$,
$\omega_n(x,y)=(x+s_n(y), y)$, and $\psi: M(v) \rightarrow M(u)$,
with $\psi_n:G_{n+1} \oplus P_n \rightarrow G_{n+1} \oplus P_n$
given by $\psi_n(x,y) = (x-s_n(y), y)$. Then $\omega_n \psi_n (x,y)
=(x,y)$ and $\psi_n \omega_n (x,y)= (x,y)$ for all $(x,y) \in
G_{n+1} \oplus P_n$.


- If $P \xrightarrow{\theta} M$, $P' \xrightarrow{\theta'} M$ are
two surjective DG-projective resolutions and $G \xrightarrow{\phi}
M$ is a special Gorenstein projective precover then there are maps
of complexes $u: P \rightarrow G$ and $v: P' \rightarrow G$, such
that $\phi u = \theta$ and $\phi v = \theta'$. We show that their
mapping cones, $M(u)$ and $M(v)$, are homotopically equivalent.

Proof: Since $P \rightarrow M$ and $P' \rightarrow M$ are
DG-projective resolutions there are maps $h:P \rightarrow P'$ and
$k:P' \rightarrow P$ induced by $1_M$. Then $vh:P \rightarrow G$ and
$u:P \rightarrow G$ are both induced by $1_M$. By the above, $M(u)$
and $M(vh)$ are isomorphic. So it suffices to show that $M(v)$ and
$M(vh)$ are homotopically equivalent.

Let $P = \ldots \rightarrow P_{n+1} \xrightarrow{f_{n+1}} P_n
\xrightarrow{f_n} P_{n-1} \rightarrow \ldots$, let $P' = \ldots
\rightarrow P'_{n+1} \xrightarrow{f'_{n+1}} P'_n \xrightarrow{f'_n}
P'_{n-1} \rightarrow \ldots$, and $G = \ldots \rightarrow G_{n+1}
\xrightarrow{g_{n+1}} G_n \xrightarrow{g_n} G_{n-1} \rightarrow
\ldots$.

 Since $hk:P' \rightarrow P'$ is induced by $1_M$, $hk$ is
homotopic to $1_{P'}$.\\ So for each n there exists $s_n \in
Hom(P'_n
, P'_{n+1})$ such that $1 - h_n k_n = f'_{n+1} s_n + s_{n-1} f'_n$.\\
Similarly $kh$ is homotopic to $1_P$, so for each n there exists
$\overline{s}_n \in Hom(P_n, P_{n+1})$, such that $1 - k_n h_n =
f_{n+1} \overline{s}_n + \overline{s}_{n-1} f_n$.\\

Let $M(v)$ denote the mapping cone of $v$:\\
$M(v) = \ldots \rightarrow G_{n+2} \oplus P'_{n+1}
\xrightarrow{\delta'_{n+1}} G_{n+1} \oplus P'_n
\xrightarrow{\delta'_n} G_n \oplus P'_{n-1} \rightarrow \ldots$,
with $\delta'_n = (g_{n+1}(x)+v_n(y), -f'_n(y))$.\\
Let $M(vh)$ be the mapping cone of $vh$:\\
$M(vh) = \ldots \rightarrow G_{n+2} \oplus P_{n+1}
\xrightarrow{\delta_{n+1}} G_{n+1} \oplus P_n \xrightarrow{\delta_n}
G_n \oplus P_{n-1} \rightarrow \ldots$,
with $\delta_n = (g_{n+1}(x)+v_n h_n(y), -f_n(y))$.\\

There are maps of complexes: $\alpha: M(v) \rightarrow M(vh)$, with
$\alpha_n:G_{n+1} \oplus P'_n \rightarrow G_{n+1} \oplus P_n$,
$\alpha_n(x,y)=(x+v_{n+1} s_n(y), k_n(y))$, and $\beta : M(vh)
\rightarrow M(v)$, with $\beta_n : G_{n+1} \oplus P_n \rightarrow
G_{n+1} \oplus P'_n$, given by $\beta_n (x,y) = (x+ v_{n+1} h_{n+1}
\overline{s}_n (y) - v_{n+1} s_n h_n (y), h_n(y))$.

Then $\alpha_n \beta_n (x,y) = (x + v_{n+1} h_{n+1} \overline{s}_n
(y), k_n h_n (y)).$

Let $\eta_n: G_{n+1} \oplus P_n \rightarrow G_{n+2} \oplus P_{n+1}$,
$\eta_n(x,y) = (0, \overline{s}_n (y))$.\\
Then $(\delta_{n+1} \eta_n + \eta_{n-1} \delta_n)(x,y) =
(v_{n+1} h_{n+1} \overline{s}_n (y), (k_n h_n -1)(y)) = \alpha_n
\beta_n (x,y) - (x,y)$. So $\alpha \beta \sim 1$.\\

We have $\beta_n \alpha_n (x,y) = (x+ v_{n+1} s_n(y) + v_{n+1}
h_{n+1} \overline{s}_n k_n (y) -v_{n+1} s_n h_n k_n (y), h_n
k_n(y))$.\\
Let $\mu_n : G_{n+1} \oplus P'_n \rightarrow G_{n+2} \oplus
P'_{n+1}$, $\mu_n (x,y) = (0, (s_n + h_{n+1} \overline{s}_n k_n -
s_n h_n k_n)(y))$.\\
Then $(\delta'_{n+1} \mu_n + \mu_{n-1} \delta'_n)(x,y)=
((v_{n+1} s_n + v_{n+1} h_{n+1} \overline{s}_n k_n -
v_{n+1} s_n h_n k_n)(y), (h_n k_n -1)(y)) = \beta_n \alpha_n (x,y) -
(x,y)$.\\
Thus $\beta \alpha \sim 1$.\\
So $M(v) \sim M(vh) \simeq M(u)$.    \hspace{65mm} (2)

- Similarly, if $P \rightarrow M$ is a surjective DG-projective
resolution, $G \rightarrow M$ and $G' \rightarrow M$ are special
Gorenstein projective precovers, and $u:P \rightarrow G$ and $v: P
\rightarrow
G'$ are induced by $1_M$, then $M(u) \sim M(v)$. \hspace{21mm} (3)\\

- If $P \rightarrow M$, $P' \rightarrow M$ are surjective
DG-projective resolutions, $G \rightarrow M$, $G' \rightarrow M$ are
special Gorenstein projective precovers, and $u: P \rightarrow G$,
$v: P' \rightarrow G'$ are maps of complexes induced by $1_M$, then
$M(u) \sim M(v)$.

Proof: There are maps of complexes $h:P \rightarrow P'$, $l: G'
\rightarrow
G$, both induced by $1_M$.\\
By (2), $M(lv) \sim M(u)$, and by (3), $M(v) \sim M(lv)$. So $M(u) \sim M(v)$.\\

\vspace{17mm} We denote by $Ext_R (-,-)$ the right derived functors
of $\mathcal{H}om(-,-)$ (the absolute cohomology). We show that over
Gorenstein rings there is a close connection between the absolute,
the Gorenstein and the generalized
Tate cohomology:\\

\begin{proposition}
Let $R$ be a Gorenstein ring, and let $M$ be a complex of
$R$-modules. For each complex $N$ of $R$-modules there is an exact
sequence $ \ldots \rightarrow Ext_R ^ {n-1} (M, N) \rightarrow
\overline{Ext}^{n-1} (M, N)
\rightarrow Ext^n_{\G }(M, N) \rightarrow Ext^n _R(M,N) \rightarrow \overline{Ext}^n(M,N) \rightarrow \ldots$.\\
\end{proposition}

\begin{proof}
Let $P \xrightarrow{\beta} M$ be a surjective DG-projective
resolution and let $G \xrightarrow{\alpha} M$ be a special
Gorenstein projective precover. $P$ is Gorenstein projective, so
there is a map of complexes $u:P \rightarrow G$
such that $\beta = \alpha u$.\\
Since the sequence $0 \rightarrow G \rightarrow M(u) \rightarrow
P[1] \rightarrow 0$ is split exact in each degree, for each complex
$N$ we have an exact sequence $0 \rightarrow \mathcal{H}om((P[1], N)
\rightarrow \mathcal{H}om (M(u), N) \rightarrow \mathcal{H}om(G, N)
\rightarrow 0$. This gives a long exact sequence $ \ldots
\rightarrow H^n \mathcal{H}om(P[1], N) \rightarrow H^n \mathcal{H}om
(M(u), N) \rightarrow H^n \mathcal{H}om (G, N) \rightarrow H^{n+1}
\mathcal{H}om (P[1], N) \rightarrow \ldots$, that is \\
$ \ldots \rightarrow Ext_R ^{n-1} (M, N) \rightarrow
\overline{Ext}^{n-1} (M, N)
\rightarrow Ext^n_{\G} (M, N) \rightarrow Ext_R ^n(M,N) \rightarrow \ldots$.\\
\end{proof}

\begin{remark}
Let $R$ be a Gorenstein ring. If $M$ and $N$ are $R$-modules
regarded as complexes at zero, then the exact sequence above gives
the Avramov-Martsinkovsky exact sequence connecting the absolute,
the Gorenstein relative and the Tate cohomology of modules (see
\cite{avramov:02:absolute}, Theorem 7.1,
\cite{veliche:04:gorenstein}, Theorem 6.6, or
\cite{iacob:04:generalized},Corollary 1):\\
 $ 0 \rightarrow Ext^1_{\G}(M,N) \rightarrow Ext^1_R(M,N) \rightarrow
\widehat{Ext}^1 (M,N) \rightarrow \ldots $
\end{remark}

\begin{proof}
Let $\ldots \rightarrow P_1 \rightarrow P_0 \rightarrow M
\rightarrow 0$ be a projective resolution and let $ \ldots
\rightarrow G_1 \xrightarrow{g_1} G_0 \xrightarrow{g_0} M
\rightarrow M \rightarrow 0$ be a special Gorenstein projective
resolution of $M$ (i.e. $G_0 \rightarrow M$, $G_i \rightarrow Ker
g_{i-1}$ are Gorenstein projective precovers and $Ker g_i$ has
finite projective dimension for each $i \ge 0$). Let $P = \ldots
\rightarrow P_1 \rightarrow P_0 \rightarrow 0$ and $G = \ldots
\rightarrow G_1 \rightarrow G_0 \rightarrow 0$. Then $P \rightarrow
(0 \rightarrow M \rightarrow 0)$ is a surjective DG-projective
resolution, and $G \rightarrow (0 \rightarrow M \rightarrow 0)$ is a
special Gorenstein projective resolution. If $u: P \rightarrow G$ is
induced by $1_M$ then by \cite{iacob:04:generalized} Proposition 1,
the cohomology modules $H^n \mathcal{H}om (M(u), N)$ are the usual
Tate cohomology modules $\widehat{Ext} ^n _R (M,N)$ for $n>0$; by
Remark 5, $H^n \mathcal{H}om (G,N)$ are the usual Gorenstein
cohomology groups.
\end{proof}

 We show that if $R$ is  a Gorenstein ring then a $Hom (Gor Proj, -)$ exact sequence of complexes $0 \rightarrow M'
 \rightarrow M \rightarrow M" \rightarrow 0$  gives a long exact
 sequence $ \ldots \rightarrow \overline {Ext}^n(M", -) \rightarrow \overline {Ext}^n(M, -)
 \rightarrow \overline {Ext}^n(M', -) \rightarrow \overline {Ext}^{n+1}(M",
  -) \rightarrow \ldots$.

  We will use the following result (the Horseshoe lemma for Gorenstein precovers).\\

  \begin{lemma}(Horseshoe Lemma)
  Let $R$ be a Gorenstein ring. If $0 \rightarrow M' \xrightarrow
  {l} M \xrightarrow{h} M" \rightarrow 0$ is a $Hom (Gor Proj,-)$ exact sequence of
  complexes then there is an exact sequence $0 \rightarrow F'
  \rightarrow F' \oplus F" \rightarrow F" \rightarrow 0$ with $F' \rightarrow
  N'$, $F' \oplus F" \rightarrow N$ and $F" \rightarrow N"$ special Gorenstein
  projective precovers.
  \end{lemma}

 \begin{proof}

 Let $F' \xrightarrow{\phi'} M'$ and $F" \xrightarrow{\phi"} M"$ be
 special Gorenstein projective precovers. By hypothesis the sequence
 $Hom(F", M) \rightarrow Hom(F", M") \rightarrow 0$ is exact. So
 there exists $u \in Hom (F", M)$ such that $\phi" = hu$.\\
 Let $\phi: F' \oplus F" \rightarrow M$ be given by $\phi_n(x,y) =
 l_n \phi'_n (x) + u_n(y)$.\\
 We have a commutative diagram
\[
\begin{diagram}
\node{0}\arrow{e}\node{F'}\arrow{s,r}{\phi'}\arrow{e} \node{F'\oplus
F''}\arrow{s,r}{\phi}\arrow{e}\node{F''}\arrow{s,r}{\phi''}\arrow{e}\node{0}\\
\node{0}\arrow{e}\node{M'}\arrow{s}\arrow{e,t}{l}\node{M}\arrow{s}\arrow{e,t}{h}\node{M''}
\arrow{s}\arrow{e}\node{0}\\
\node{}\node{0}\node{0}\node{0}
\end{diagram}
\]


The exact sequence $0 \rightarrow Ker \phi' \rightarrow Ker \phi
\rightarrow Ker \phi" \rightarrow 0$ with both $Ker \phi'$ and $Ker
\phi"$ of finite projective dimension gives that $Ker \phi$ has
finite projective dimension. So $F' \oplus F" \xrightarrow{\phi} M$
is a special Gorenstein projective precover.

 \end{proof}

 \begin{proposition}
 Let $0 \rightarrow M' \xrightarrow
  {l} M \xrightarrow{h} M" \rightarrow 0$ be a $Hom (Gor Proj, -)$ exact sequence. Then for each
  complex $N$ we have an exact sequence $ \ldots \rightarrow
  Ext_{\G} ^n (M", N) \rightarrow  Ext_{\G} ^n (M, N) \rightarrow  Ext_{\G} ^n (M', N) \rightarrow Ext_{\G} ^{n+1} (M",
  N) \rightarrow \ldots$
  \end{proposition}

  \begin{proof}
By the Horseshoe Lemma there exists an exact sequence $0 \rightarrow
F' \rightarrow F' \oplus F" \rightarrow F" \rightarrow 0$ with $F'
\rightarrow M'$, $F' \oplus F" \rightarrow M$ and $F" \rightarrow
M"$ special Gorenstein projective precovers. For each complex $N$ we
have an exact sequence $0 \rightarrow \mathcal{H}om (F", N)
\rightarrow \mathcal{H}om (F' \oplus F", N) \rightarrow
\mathcal{H}om (F', N) \rightarrow 0$ and therefore an exact
sequence: $\ldots \rightarrow H^n \mathcal{H}om (F", N) \rightarrow
H^n \mathcal{H}om (F'\oplus F", N) \rightarrow H^n \mathcal{H}om
(F', N) \rightarrow \ldots$.

  \end{proof}

The same argument as in Lemma 5 gives:\\
\begin{lemma}
If $0 \rightarrow M' \rightarrow M \rightarrow M" \rightarrow 0$ is
a $Hom (DG Proj, -)$ exact sequence then there exists an exact
sequence $0 \rightarrow P' \rightarrow P' \oplus P" \rightarrow P"
\rightarrow 0$ with $P' \rightarrow M'$, $P' \oplus P" \rightarrow
M$, and $P" \rightarrow M"$ surjective DG-projective resolutions.
\end{lemma}

\begin{proposition}
If $0 \rightarrow M' \rightarrow M \rightarrow M" \rightarrow 0$ is
a $Hom (Gor Proj, -)$ exact sequence then for each complex $N$ we
have an exact sequence: $\ldots \rightarrow \overline{Ext}^n (M", N)
\rightarrow \overline{Ext}^n (M, N) \rightarrow \overline{Ext}^n
(M', N) \rightarrow \overline{Ext}^{n+1} (M", N) \rightarrow
\ldots$.
\end{proposition}

\begin{proof}
By Lemma 5 and Lemma 6 we have commutative diagrams:

\[
\begin{diagram}
\node{0}\arrow{e}\node{P'}\arrow{s,r}{\phi'}\arrow{e,t}{j}
\node{P'\oplus P''}\arrow{s,r}{\phi}\arrow{e,t}{\pi}
\node{P''}\arrow{s,r}{\phi''}\arrow{e}\node{0}\\
\node{0}\arrow{e}\node{M'}\arrow{e,t}{l}\node{M}\arrow{e,t}{h}
\node{M''}\arrow{e}\node{0}
\end{diagram}
\]

\[
\begin{diagram}
\node{0}\arrow{e}\node{G'}\arrow{s,r}{\tau'}\arrow{e,t}{i}
\node{G'\oplus G''}\arrow{s,r}{\tau}\arrow{e,t}{p}
\node{G''}\arrow{s,r}{\tau''}\arrow{e}\node{0}\\
\node{0}\arrow{e}\node{M'}\arrow{e,t}{l}\node{M}\arrow{e,t}{h}
\node{M''}\arrow{e}\node{0}
\end{diagram}
\]



with $P' \xrightarrow{\phi'} M'$, $P'\oplus P"  \xrightarrow{\phi}
M$, $P" \xrightarrow{\phi"} M"$ surjective DG-projective
resolutions, and with $G' \xrightarrow{\tau'} M'$, $G' \oplus G"
\xrightarrow{\tau} M$ and $G" \xrightarrow{\tau"} M"$ special
Gorenstein projective precovers.\\
Since $P'$ is Gorenstein projective there exists a map of complexes
$u: P' \rightarrow G'$ such that $\tau' u = \phi'$.\ Let $\alpha: P"
\rightarrow M$ be given by $\alpha_j(y) = \phi_j(0,y)$. Since $P"$
is Gorenstein projective and $G' \oplus G" \xrightarrow{\tau} M$ is
a Gorenstein projective precover there exists $\beta: P" \rightarrow
G' \oplus G"$ such that $\tau \beta = \alpha$.\\
Let $\omega : P' \oplus P" \rightarrow G' \oplus G"$ be defined by
$\omega_j(x,y) = (u_j(x), 0) + \beta_j(y)$.

We have an exact sequence of complexes $0 \rightarrow M(u)
\xrightarrow{(i, j)} M(\omega) \xrightarrow{(p, \pi)} M(p \beta)
\rightarrow 0$ (where $p: G' \oplus G" \rightarrow G"$, $p_n (x,y)
=y$, $\pi: P' \oplus P" \rightarrow P"$, $\pi_n(z,t)=t$, $i:G'
\rightarrow G' \oplus G"$, $i_n(x) = (x,0)$, $j: P' \rightarrow P'
\oplus P"$, $j_n
(y) = (y,0)$).\\
The sequence is split exact in each degree, so for each complex $N$
there is an exact sequence $0 \rightarrow \mathcal{H}om (M(p \beta),
N) \rightarrow \mathcal{H}om (M(\omega), N) \rightarrow
\mathcal{H}om (M(u), N) \rightarrow 0$, and therefore an exact
sequence\\ $\ldots \rightarrow \overline{Ext}^n (M", N) \rightarrow
\overline{Ext}^n (M,N) \overline{Ext}^n (M', N') \rightarrow
\overline{Ext}^ {n+1} (M,N) \rightarrow \ldots$.

\end{proof}

In \cite{veliche:04:gorenstein} Veliche defined Tate cohomology functors for complexes of finite Gorenstein projective dimension over
arbitrary rings.\\

\begin{definition}(\cite{veliche:04:gorenstein}, Definition 4.1)
Let $M$ be a complex of finite Gorenstein projective dimension, let
$T \rightarrow P \rightarrow M$ be a complete resolution of $M$ and
let $N$ be an arbitrary complex. For each integer $n$, the $n$th
Tate cohomology group is defined by $\widehat{Ext} ^n _R (M, N) =
H^n \mathcal{H}om (T, N)$.
\end{definition}

 Let
$R$ be a Gorenstein ring. We show that for a bounded complex $M: 0
\rightarrow M_n \rightarrow \ldots \rightarrow M_0 \rightarrow 0$,
we have $\overline{Ext} ^j (M,N) \simeq \widehat{Ext}^j (M,N) $ for
any  $j
> n$, for any module $N$.

We show first\\
\begin{lemma}  Let $R$ be a Gorenstein ring. A bounded complex $M(n) = 0 \rightarrow M_n \rightarrow \ldots
\rightarrow M_0 \rightarrow 0$ has a Gorenstein projective precover
$G(n) \rightarrow M(n)$ with $G(n) = 0 \rightarrow G_k
\xrightarrow{g_k} \ldots \rightarrow G_0 \rightarrow 0$, such that
$G_i$ is projective 
for $i> n$,
 such that $L^n=Ker(G(n) \rightarrow M(n))$ is a complex
of finite projective dimension, and $L_k^{n+1} = L_k^n$ for $0 \le k
\le n$.\end{lemma}

\begin{proof}

Proof by induction on n:

Case $n=0$. Then $M=0 \rightarrow M_0 \rightarrow 0$\\
Let $ 0 \rightarrow G_k \rightarrow G_{k-1} \xrightarrow{g_k} \ldots
\rightarrow G_0 \xrightarrow{g_0}  M_0 \rightarrow 0$ be a special
Gorenstein projective resolution of $M_0$, 
and let $G = \ldots \rightarrow G_1 \rightarrow G_0 \rightarrow 0$.

Then $G \rightarrow M$ is a special Gorenstein projective
precover 
and $Ker (G \rightarrow M)$ is a complex of finite
projective dimension.

Case $ n \rightarrow n+1$. Let $M = 0 \rightarrow M_{n+1}
\xrightarrow{l_{n+1}} M_n \rightarrow
\ldots M_1 \xrightarrow{l_1} M_0 \rightarrow 0$.\\
There is a map of complexes:

\[
\begin{diagram}
\node{M':}\arrow{s,r}{l}\node{0}\arrow{e}
\node{M_{n+1}}\arrow{s,r}{l_{n+1}}\arrow{e}\node{0}\\
\node{\overline{M}:}\node{0}\arrow{e}
\node{M_n}\arrow{e,t}{l_n}\node{\ldots}\arrow{e}\node{M_1}
\arrow{e,t}{l_1}\node{M_0}\arrow{e}\node{0}
\end{diagram}
\]

By induction hypothesis there is a special Gorenstein projective
precover $\overline{G} \xrightarrow{\alpha} \overline{M}$ where $
\overline{G} =  \ldots  \xrightarrow{g_{n+1}} G_n \xrightarrow{g_n}
G_{n-1} \rightarrow \ldots \xrightarrow{g_0} G_0 \rightarrow 0$ is a
bounded complex with $G_i$ projective for $i
\ge n+1$.\\ 
Let $G' = \ldots \rightarrow G'_0 \xrightarrow{g'_0} M_{n+1}
\rightarrow 0$ be a special Gorenstein projective resolution of
$M_{n+1}$. Then $G_i$ is projective for all $k \ge 1$. The map $M'
\xrightarrow{l} \overline{M}$ induces a map of complexes $u: G'
\rightarrow \overline{G}$.\\


\[
\begin{diagram}
\node{\ldots}\arrow{e}\node{G'_1}\arrow{s,r}{u_1}\arrow{e}\node{G'_0}
\arrow{s,r}{u_0}\arrow{e}\node{0}\arrow{e}\node{0}\\
\node{\ldots}\arrow{e}\node{G_{n+1}}\arrow{e}\node{G_n}\arrow{e,t}{g_n}
\node{G_{n-1}}\arrow{e}\node{G_{n-2}}\arrow{e}\node{\ldots}
\end{diagram}
\]

such that the diagram\\
\[
\begin{diagram}
\node{G'}\arrow{s,r}{u}\arrow{e,t}{g}\node{M'}\arrow{s,r}{l}\\
\node{\overline{G}}\arrow{e,t}{\alpha}\node{\overline{M}}
\end{diagram}
\]

is commutative. In particular, $\alpha_n u_0 = l_{n+1} g'_0$.\\



Let $G$ be the mapping cone of $u$:\\
$G= \ldots  \rightarrow G_{n+1} \oplus G'_0
\xrightarrow{\overline{g}_{n+1}} G_n \xrightarrow{g_n} G_{n-1}
\rightarrow \ldots \rightarrow G_0 \rightarrow 0$ with
$\overline{g}_{n+1} (x, y) = g_{n+1} (x) +u_0(y)$ and
$\overline{g}_{n+k}(x,y)=(g_{n+k}(x)+u_k(y), -g'_k(y))$ for $k \ge
1$.

We show that $G \rightarrow M$ is a special Gorenstein projective
precover (with $G_i$ projective for $i > n+1$).

Let $A$ be a Gorenstein projective complex and let $\omega : A
\rightarrow M$; this gives a map of complexes $ \omega: \overline{A}
\rightarrow \overline{M}$ where $\overline{A} = 0 \rightarrow A_n
\rightarrow \ldots \rightarrow A_0 \rightarrow 0$.

\[
\begin{diagram}
\node{0}\arrow{e}\node{A_n}\arrow{s,r}{\omega_n}\arrow{e}\node{A_{n-1}}\arrow{s,r}{\omega_{n-1}}\arrow{e}
\node{\ldots}\arrow{e}\node{A_1}\arrow{s,r}{\omega_1}\arrow{e}\node{A_0}\arrow{s,r}{\omega_0}\arrow{e}
\node{0}\\
\node{0}\arrow{e}\node{M_n}\arrow{e}\node{M_{n-1}}\arrow{e}\node{\ldots}\arrow{e}\node{M_1}\arrow{e}
\node{M_0}\arrow{e}\node{0}
\end{diagram}
\]

Since $\overline{G} \xrightarrow{\alpha} \overline{M}$ is a
Gorenstein projective precover there exists a map $\gamma:
\overline{A} \rightarrow \overline{G}$ such that $\alpha \circ
\gamma =
\omega$.\\
Since $G'_0 \xrightarrow{g'_0} M_{n+1}$ is a Gorenstein projective
precover and $A_{n+1}$ is a Gorenstein projective module, there
exists $t_0:A_{n+1} \rightarrow G'_0$ such that $g'_0 t_0 =
\omega_{n+1}$.




\[
\begin{diagram}
\node{\ldots}\arrow{e}\node{A_{n+2}}\arrow{e,t}{a_{n+2}}\node{A_{n+1}}
\arrow{s,r}{(0,t_0)}\arrow{e,t}{a_{n+1}}\node{A_n}\arrow{s,r}{\gamma_n}\arrow{e}
\node{\ldots}\\
\node{\ldots}\arrow{e}\node{G_{n+2}\oplus
G'_1}\arrow{e,t}{{\overline{g}}_{n+2}}\node{G_{n+1}\oplus
G'_0}\arrow{s,r}{(0,g'_0)}\arrow{e,t}{\overline{g}_{n+1}}\node{G_n}\arrow{s,r}{\alpha_n}\arrow{e,t}{g_n}
\node{\ldots}\\
\node{}\node{0}\arrow{e}\node{M_{n+1}}\arrow{e,t}{l_{n+1}}\node{M_n}\arrow{e}\node{\ldots}
\end{diagram}
\]

Then $(0, g'_0)(0, t_0) = \omega_{n+1}$ and $\alpha_n
\overline{g}_{n+1} =
l_{n+1} (0,g'_0)$.\\
 Both $\overline{G}$ and
$G'$ are bounded, so $G$ is also a bounded complex. Since
$H_j(\overline{G})=0 = H_j(G')$ for any $j \ge n+1$ it follows that
$H_j(G)=0$ for $j \ge n+1$. Then the complex $ \ldots \rightarrow
G_{n+3} \oplus G'_2 \xrightarrow{\overline{g}_{n+3}} G_{n+2} \oplus
G'_1 \xrightarrow{\overline{g}_{n+2}} Im \overline{g}_{n+2}
\rightarrow 0$ is exact, bounded, with $G_{n+k} \oplus G'_{k-1}$
projective for $k \ge 2$. It follows that $Im \overline{g}_j$ has
finite projective dimension for any $j
> n+1$.

Since $g_n(\gamma_n a_{n+1} - \overline{g}_{n+1}(0,t_0)) = 0$ and
also, $\alpha_n (\gamma_n a_{n+1} - \overline{g}_{n+1}(0,t_0)) =
\omega_n a_{n+1} - l_{n+1} (0, g'_0)(0,t_0)=0$, 
we have $\gamma_n a_{n+1} - \overline{g}_{n+1}(0,t_0): A_{n+1}
\rightarrow Ker g_n \bigcap Ker \alpha_n = Ker g_n \bigcap L_n^n$.
\hspace{69mm} (4)

By induction hypothesis $L^n= \ldots L_{n+1}^n \xrightarrow{g_{n+1}}
L_n^n \xrightarrow{g_n} L_{n-1}^n \rightarrow \dots$
has finite projective dimension. 
 Since $A_{n+1}$ is Gorenstein projective, the complex $Hom(A_{n+1},
L^n)$ is exact. By (4), $\gamma_n a_{n+1} -
\overline{g}_{n+1}(0,t_0)= g_{n+1}h$ for some $h:A_{n+1} \rightarrow
L_{n+1} \subset G_{n+1}$. Thus $\gamma_n a_{n+1} =
 \overline{g}_{n+1}
\gamma_{n+1}$ with $\gamma_{n+1}: A_{n+1} \rightarrow G_{n+1} \oplus
G'_0$, $\gamma_{n+1}=(h, t_0)$. 

We have $\gamma_{n+1} a_{n+2}: A_{n+2} \rightarrow Ker
\overline{g}_{n+1} = Im \overline{g}_{n+2}$. Since the complex:\\
$ \ldots \rightarrow G_{n+k} \oplus G'_{k-1} \rightarrow \ldots
\rightarrow G_{n+2} \oplus G'_1 \xrightarrow{\overline{g}_{n+2}} Im
\overline{g}_{n+2} \rightarrow 0$\\ has finite projective dimension
and $A_{n+2}$ is Gorenstein projective, it follows that
 $ \ldots
 \rightarrow
Hom(A_{n+2}, G_{n+2} \oplus G'_1) \rightarrow Hom(A_{n+2}, Im
g_{n+2}) \rightarrow 0$ is exact. Then $\gamma_{n+1} a_{n+2} =
\overline{g}_{n+2} \gamma_{n+2}$ for some $\gamma_{n+2} \in
Hom(A_{n+2}, G_{n+2} \oplus G'_1)$. Similarly there exists
$\gamma_{n+k}$ such that $\gamma_{n+k}
a_{n+k+1} = \overline{g}_{n+k+1} \gamma_{n+k+1}$ for $k \ge 2$.\\
Since $L^{n+1} = \ldots G_{n+2} \oplus G'_1
\xrightarrow{\overline{g}_{n+2}} G_{n+1} \oplus Ker g'_0 \rightarrow
Ker \alpha_n \rightarrow \ldots \rightarrow Ker \alpha_0 \rightarrow
0$ and $L^n = \ldots \rightarrow G_{n+2} \rightarrow G_{n+1}
\rightarrow Ker \alpha_n \rightarrow \ldots \rightarrow Ker \alpha_0
\rightarrow 0$ we have $L^n_k = L^{n+1}_k$ for $0 \le k \le n$.\\
Since $\frac{L^{n+1}}{L^n} \simeq \ldots \rightarrow G'_2
\xrightarrow{g'_2} G'_1 \xrightarrow{g'_1} Ker g'_0 \rightarrow 0$,
the module $Ker g'_0$ has finite projective dimension and $G'_j$ is
projective for each $j \ge 1$, it follows that $\frac{L^{n+1}}{L^n}$
has finite projective dimension. By induction hypothesis $L^n$ has
finite projective dimension. The exact sequence $0 \rightarrow L^n
\rightarrow L^{n+1} \rightarrow \frac{L^{n+1}}{L^n} \rightarrow 0$
gives that $L^{n+1}$ has finite projective dimension.
\end{proof}

\begin{proposition}
Let $R$ be a Gorenstein ring. If $M=0 \rightarrow M_n \rightarrow
M_{n-1} \rightarrow \ldots \rightarrow M_1 \rightarrow M_0
\rightarrow 0$ is a bounded complex then $\overline{Ext}^j (M,N)
\simeq \widehat{Ext}^j (M, N)$ for $j > n$, for any module $N$.
\end{proposition}

\begin{proof}
By Lemma 7 there is a special Gorenstein projective precover $G
\rightarrow M$ with $G= 0 \rightarrow G_g \rightarrow \ldots
\rightarrow G_0 \rightarrow 0$ a bounded complex and such that $G_j$
is projective for all $j \ge n+1$. By \cite{veliche:04:gorenstein},
1.3.4, there exists a surjective DG-projective resolution $P
\rightarrow M$ with $P= \ldots \rightarrow P_1 \rightarrow P_0
\rightarrow 0$. Since $P$ is Gorenstein projective there is a map of
complexes $u : P \rightarrow G$ such that the diagram

\[
\begin{diagram}
\node{P=\ldots}\arrow{s}\arrow{e}\node{P_{g+1}}\arrow{e}\node{P_g}
\arrow{s,r}{u_g}\arrow{e,t}{f_g}\node{\ldots}\arrow{e}
\node{P_1}\arrow{s,r}{u_1}\arrow{e}\node{P_0}\arrow{s,r}{u_0}\arrow{e}\node{0}\\
\node{G=\ldots}\arrow{e}\node{0}\arrow{e}\node{G_g}\arrow{e,r}{g_g}
\node{\ldots}\arrow{e}\node{G_1}\arrow{e}\node{G_0}\arrow{e}\node{0}
\end{diagram}
\]

is commutative.

Let $M(u) = \ldots  \rightarrow  P_g \xrightarrow{f_g} G_g \oplus
P_{g-1} \xrightarrow{\delta_{g-1}} \ldots \rightarrow G_1 \oplus P_0
\xrightarrow{\delta_0} G_0 \rightarrow 0$ be the mapping cone.

The sequence $0 \rightarrow G \rightarrow M(u) \rightarrow P[1]
\rightarrow 0$ is exact, $H_n(G) = H_n (M) =H_n(P)$ for all n, so
$M(u)$ is exact. Since $M(u)$ is a right bounded exact complex of
Gorenstein projective modules it follows that $Ker \delta_j$ is
Gorenstein projective for
all j.\\
Let $\overline{M} = \ldots \rightarrow G_{n+2} \oplus P_{n+1}
\rightarrow G_{n+1} \oplus P_n \rightarrow Im \delta_n \rightarrow
0$. Each $P_l$ is projective and $G_j$ is projective for $j>n$, so
 $G_j
\oplus P_{j-1}$ is projective for any $j > n$.\\
Since $Im \delta_n$ is a Gorenstein projective module, there is a
$Hom (-, Proj)$ exact exact sequence $0 \rightarrow Im \delta_n
\rightarrow
U_{n-1} \rightarrow U_{n-2} \rightarrow \ldots$ with each $U_j$ a projective module..\\
Let $M' = \ldots G_{n+1} \oplus P_n \rightarrow U_{n-1} \rightarrow
U_{n-2} \rightarrow \ldots$ Then
$M'$ is an exact complex of projective modules that is also $Hom(-, Proj)$ exact.\\

The map of complexes $M(u) \rightarrow P[1]$ gives an
$R$-homomorphism $\omega: Im \delta_n \rightarrow Im f_n$. Since $0
\rightarrow Im \delta_n \rightarrow U_{n-1} \rightarrow U_{n-2}
\rightarrow \ldots$ is a $Hom(-, Proj)$ exact sequence there are
homomorphisms $\omega_j : U_j \rightarrow P_j$ for all $j \le
n-1$ such that the diagram\\

\[
\begin{diagram}
\node{0}\arrow{e}\node{Im\;\delta_n}\arrow{s,r}{\omega}\arrow{e,t}{i}
\node{U_{n-1}}\arrow{s,r}{\omega_{n-1}}\arrow{e,t}{u_n}\node{U_{n-2}}\arrow{s,r}{\omega_{n-2}}\arrow{e}\node{\ldots}\\
\node{0}\arrow{e}\node{Im\;f_n}\arrow{e,t}{j}\node{P_{n-1}}\arrow{e,b}{f_{n-1}}\node{P_{n-2}}\arrow{e}\node{\ldots}
\end{diagram}
\]

is commutative.\\

So there exists an exact complex $M'$ of projective modules that is
$Hom (-, Proj)$ exact, and a map of complexes $M' \xrightarrow{\pi}
P [1]$, with $\pi_j = \omega_j$ for $j \le n-1$, $\pi_l (x,y) = y$
for $l \ge n$ and with $\pi_j = 1_j$ for all $j \ge g+1$ (because
$G_j =0$ for $j \ge g+1$). Then $M'[-1] \rightarrow P \rightarrow M$
is a complete resolution.

Since $M'_j = M(u)_j$ for $j \ge n+1$ it follows that
$\widehat{Ext}^j(M,N) \simeq \overline{Ext}^j (M,N)$ for any $j \ge
n+1$, for any module $N$.

\end{proof}



\begin{thebibliography}{99}
\bibitem{asadollahi:06:gicomplexes}
J.~Asadollahi and S.~Salarian.
\newblock {Gorenstein injective dimension for complexes and Iwanaga-Gorenstein
  rings}.
\newblock {\em Comm. Alg.}, (34):3009--3022, 2006.

\bibitem{avramov:91:homological}
L.~Avramov and H.-B. Foxby.
\newblock Homological dimensions of unbounded complexes.
\newblock {\em J. Pure Appl. Algebra}, (71):129--155, 1991.

\bibitem{avramov:02:absolute}
L.L. Avramov and A.~Martsinkovsky.
\newblock {Absolute, Relative and Tate cohomology of modules of finite
  Gorenstein dimension}.
\newblock {\em Proc. London Math. Soc.}, 3(85):393--440, 2002.

\bibitem{bennis:08:rings}
D.~Bennis.
\newblock {Rings over which the class of Gorenstein flat modules is closed
  under extensions}.
\newblock {\em {Communications in Algebra}}, 37(3):855--868, 2009.

\bibitem{christensen:00:gorenstein}
L.W. Christensen.
\newblock {\em {Gorenstein dimensions}}, volume 1747 of {\em Lecture Notes in
  Math.}
\newblock Springer, Berlin, 2000.

\bibitem{christensen:08:beyound}
L.W. Christensen, H-B. Foxby, and H.~Holm.
\newblock {Beyond Totally Reflexive Modules and Back}.
\newblock Springer-Verlag, to appear.

\bibitem{christensen:06:ongorenstein}
L.W. Christensen, A.~Frankild, and H.~Holm.
\newblock On gorenstein projective, injective and flat dimensions -- a
  functorial description with applications.
\newblock {\em Journal of Algebra}, 302:231--279, 2006.

\bibitem{enochs:00:relative}
E.E. Enochs and O.M.G. Jenda.
\newblock {\em Relative Homological Algebra}.
\newblock Walter de Gruyter, 2000.
\newblock De Gruyter Exposition in Math; 30.

\bibitem{enochs:93:gorenstein}
E.E. Enochs, O.M.G. Jenda, and B.~Torrecillas.
\newblock {Gorenstein flat modules}.
\newblock {\em Journal Nanjing Univ.}, 10:1--9, 1993.

\bibitem{enochs:96:orthogonality}
E.E. Enochs, O.M.G. Jenda, and J.~Xu.
\newblock Orthogonality in the category of complexes.
\newblock {\em Math. J. Okayama Univ.}, 38:25--46, 1996.

\bibitem{holm:04:gorenstein}
H.~Holm.
\newblock Gorenstein homological dimensions.
\newblock {\em J. Pure and Appl. Alg.}, 189:167--193, 2004.

\bibitem{iacob:04:generalized}
A.~Iacob.
\newblock {Generalized Tate cohomology}.
\newblock {\em Tsukuba Journal of Mathematics}, 29(2):389--404, 2005.

\bibitem{garcia:99:covers}
J.R.~Garc{\'i}a Rozas.
\newblock {\em Covers and evelopes in the category of complexes of modules}.
\newblock CRC Press LLC, 1999.

\bibitem{veliche:04:gorenstein}
O.~Veliche.
\newblock {Gorenstein projective dimension for complexes}.
\newblock {\em Trans. Amer. Math. Soc}, (358):1257--1283, 2006.

\bibitem{yassemi:95:gorenstein}
S.~Yassemi.
\newblock Gorenstein dimensions.
\newblock {\em Math. Scand.}, 77:161--174, 1995.

\end{thebibliography}
\end{document}